\documentclass[a4paper, 11pt]{article}

\usepackage{amsmath}
\usepackage{amsfonts}
\usepackage{amssymb}
\usepackage[english]{babel}
\usepackage{graphicx}
\usepackage{amsthm}
\usepackage{mathrsfs}
\usepackage{pgf,tikz}
\usepackage{amstext} 
\usepackage{array}   
\newcolumntype{C}{>{$}c<{$}} 
\definecolor{uququq}{rgb}{0.25,0.25,0.25}
\newtheorem{thm}{Theorem}[section]
\newtheorem{cor}[thm]{Corollary}

\usetikzlibrary{arrows}
\theoremstyle{definition}

\theoremstyle{definition}

\theoremstyle{definition}

\newcommand{\Z}{\mathbb{Z}}

\newcommand{\F}{\mathbb{F}}

\newcommand{\comment}[1]{}
\numberwithin{equation}{section}

\begin{document}
\date{}
\title{Partitioned difference families:\\the storm has not yet passed}
\author{Marco Buratti \thanks{Dipartimento di Matematica e Informatica, Universit\`a di Perugia, via Vanvitelli 1, Italy. Email: buratti@dmi.unipg.it}\\
\\
Dieter Jungnickel
\thanks{Mathematical Institute, University of Augsburg, 86135 Augsburg, Germany.
Email: jungnickel@math.uni-augsburg.de}}
\date{\today}
\maketitle
\begin{abstract}
Two years ago, we alarmed the scientific community about the large number of bad papers in the literature on {\it zero difference balanced functions},
where direct proofs of seemingly new results are presented in an unnecessarily lengthy and convoluted way.
Indeed, these results had been proved long before and very easily in terms of difference families.

In spite of our report, papers of the same kind continue to proliferate. Regrettably, a  further attempt to put the topic in order seems unavoidable. 
While some authors now follow our recommendation of using the terminology of {\it partitioned difference families},  
their methods are still the same and their results are often trivial or even wrong.
In this note, we show how a very recent paper of this type can be easily dealt with.
\end{abstract}

\section{Introduction}

We recall that a collection $\cal F$ of subsets ({\it blocks}) of an additive group $G$ is a {\it difference family} (DF) of index $\lambda$ if the list of differences from $\cal F$, 
that is the multiset $\Delta{\cal F}=\{x-y \ | \ (x,y)\in B\times B; B\in{\cal F}\}$, covers every non-zero element of $G$ exactly $\lambda$ times. 
If $G$ has order $v$ and $K$ is the multiset of all the block-sizes, then one says that ${\cal F}$ is a $(v,K,\lambda)$-DF in $G$. 
If all blocks have  size $k$, the DF is said to be {\it uniform}, and one writes $(v,k,\lambda)$ rather than $(v,K,\lambda)$.
If we have only one block, then this block is said to be a $(v,k,\lambda)$ {\it difference set} (DS). 
For general background on uniform DFs and DSs we refer to \cite{AB} and \cite{DPS}, respectively. 
A DF whose blocks are pairwise disjoint is said to be {\it disjoint} (DDF), and it is {\it partitioned} (PDF) if the blocks partition $G$. 
It is evident that every DDF can be extended to a PDF by adding all possible singletons $\{g\}$ which are not contained in one of its blocks. 
From the design theory perspective, the PDFs having one block of size $k$ and all  other blocks  of size $k+1$ are of particular interest since they are equivalent to
{\it resolvable Steiner $2$-designs} with an automorphism group acting sharply transitively on all but one point (see \cite{BYW}).
 
 The notion of a PDF was introduced in \cite{DY} in view of its application to {\it optimal constant composition codes}.
 Subsequently, the equivalent notion of a {\it zero difference balanced function} (ZDBF) has been considered, and this  probably caused some confusion. 
 In \cite{B} it has been shown that the most celebrated results on ZDBFs were known since the 90's and can be immediately deduced from the following much more general result:
 
 \begin{thm}\label{Buratti}{\rm\cite{B}}
 Let $A$ be a group of automorphisms of order $k$ of a group $G$ of order $v$. If the action of $A$ on $G\setminus\{0\}$ 
 is semiregular\footnote{That is, for $\alpha\in A$  and $g\in G\setminus\{0\}$ we have $\alpha(g)=g$  if and only if $\alpha=id_G$.}, 
 then the set $\cal F$ of $A$-orbits on  $G\setminus\{0\}$ is a $(v,k,k-1)$-DDF in $G$.
 Moreover, if $G$ is abelian and $vk$ is odd, $\cal F$ can be split into two $(v,k,{k-1\over2})$-DDFs. \qed
 \end{thm}
 
 One of the many consequences of the above result is the disjoint version of some old DFs due to Steven Furino.
In what follows, given an integer $v>1$, we shall denote the ring of order $v$ which is the direct product of finite fields by $R_v$.
By abuse of language, if we speak of a DF in $R_v$, we will mean a DF in the additive group of $R_v$.
 
 \begin{cor}\label{Furino} {\rm\cite{F}}
 If the prime divisors (resp. the maximal prime power divisors) of $v$ are all congruent to $1$ $($mod $k)$, then there exists a $(v,k,k-1)$-DDF in $\Z_v$ (resp. $R_v$).
 Under the additional hypothesis that $vk$ is odd,  there also exists a $(v,k,{k-1\over2})$-DDF in $\Z_v$ (resp. $R_v$). \qed
 \end{cor}

 In \cite{BJ}, besides giving four good reasons to prefer PDFs rather than ZDBFs, we had to point out that despite the
 work done by the first author in \cite{B},  ZDBF-papers without anything new, except for tremendously involved proofs in a 
 different setting, continued to proliferate. Anyway, as the title of the present note says, {\it the storm has not yet passed}.

For instance, a very recent 20-page paper \cite{AMC} claims to present three new classes of PDFs.
However, one of these classes is a trivial consequence of Theorem \ref{Buratti}, the second class is a special case of Corollary \ref{Furino}, and the third construction is actually wrong. 
Moreover, the corrected version of their third class turns out to be a very special case of a well-known, rather trivial, construction in design theory.

In their introduction, the authors of \cite{AMC} comment that our note \cite{BJ} deserves full attention.
Nevertheless, apart from the terminology, they continue to apply {\it generalized cyclotomy} as in almost all papers dealing with ZDBFs. 
It is conceivable that this concept introduced in several papers such as \cite{Z} could have interesting applications. 
Unfortunately, none of the three applications given by the authors falls into this category.

In this note we will give a complete analysis of \cite{AMC}, with a twofold aim: to stop the flood of useless papers in this area,
and to encourage researchers to find new results on this interesting topic, after seriously studying the established literature on difference sets, difference families, and their related objects.

\section{Analysis of a recent paper on PDFs}

The three main results of the paper under investigation \cite{AMC}, displayed in its Table 1, can
be formulated as follows. 

\begin{quote}
{Result 1.} If $v$ is a product of prime powers all congruent to 1 $($mod $k(k+1))$, 
then there exists a $(v(k+1),k,k-1)$-DDF in $\Z_{k+1}\times R_v$.
\end{quote}

\begin{quote}
{Result 2.} If $k$ is odd and $v$ is a product of prime powers all congruent to 1 (mod $2k$), then there exists a
$(v,k,{k-1\over2})$-DDF in $R_v$.
\end{quote}

\begin{quote}
{Result 3.} For any prime power $q\equiv1$ $($mod $e)$, any $m\geq3$ coprime with $e$, and any $h$ in the
closed interval $[1,e]$, there exists a $(v,k,\lambda)$-DS in $\Z_{{q^m-1\over e}}\times \Z_h$ with 
\noindent\begin{center}
$v={(q^m-1)h\over e};\quad k={(q^{m-1}-1)h\over e};\quad \lambda={(q^{m-2}-1)h\over e}.$
\end{center}
\end{quote}

\subsection{Comments on Result 1}

Note that our statement of Result 1 arises from the one given by the authors by renaming $e-1$ as $k$.
Their hypotheses are much stronger than necessary, since ``(mod $k(k+1)$)" can be replaced by ``(mod $k$)"
and $\Z_{k+1}$ can be replaced by any group (even non-abelian) of order $k+1$. Also, the proof of this strengthened version is very elementary:
it is an immediate consequence of a considerably more general result easily deducible from the disjoint version of Theorem 4.1 in \cite{D}. 

Before stating and proving this result, we recall that a $(v,k,1)$ {\it difference matrix} (DM) in an additive group $H$ of order $v$ is a $k\times v$ matrix with elements in $H$ 
and the property that the difference of any two distinct rows is a permutation of $H$. 
It is called {\it homogeneous} (HDM) if each row is a permutation of $H$ as well. We note that adding a $0$-row to a $(v,k,1)$-HDM gives a $(v,k+1,1)$-DM. 
Conversely, any $(v,k+1,1)$-DM can be {\it normalized} as described in \cite{D2} to a $(v,k+1,1)$-DM with a $0$-row. Obviously, deleting this row  results in a $(v,k,1)$-HDM. 
Thus  a $(v,k,1)$-HDM is completely equivalent to a $(v,k+1,1)$-DM. 

\begin{thm}
If there exist a $(u,k,k-1)$-DDF in $G$, a $(v,k,k-1)$-DDF in $H$, and a $(v,k+1,1)$-DM, then there also exists a $(uv,k,k-1)$-DDF in $G\times H$.
\end{thm}

\begin{proof}
Let ${\cal A}$ be a $(u,k,k-1)$-DDF in $G$, let ${\cal B}$ be a $(v,k,k-1)$-DDF in $H$, and let $M$ be a $(v,k,1)$-HDM in $H$ 
(which exists in view of its equivalence with a $(v,k+1,1)$-DM). For every $A=\{a_1,a_2,\dots,a_k\}\in{\cal A}$ and each $j\in \{1,\dots, v\}$, 
consider the $k$-subset $A_j$ of $G\times H$ defined by $A_j=\{(a_1,m_{1j}),(a_2,m_{2j}),\dots,(a_k,m_{kj})\}.$ 
Now let $g$ be the unique element of $G$ not covered by the blocks of ${\cal A}$ and set
$${\cal F}=\{A_j \ | \ A\in{\cal A}; 1\leq j\leq v\} \ \cup \ \{\{g\}\times B \ | \ B\in{\cal B}\}.$$
It is straightforward to check that ${\cal F}$ is the required $(uv,k,k-1)$-DDF in $G\times H$.
\end{proof}

In view of the above theorem, in order to obtain Result 1 it is enough to have the following ingredients:
\begin{itemize}
\item[(i)] a $(k+1,k,k-1)$-DDF in $\Z_{k+1}$;
\item[(ii)] a $(v,k,k-1)$-DDF in $R_v$;
\item[(iii)] a $(v,k+1,1)$-DM in $R_v$.
\end{itemize}

The first ingredient is the trivial difference set $\Z_{k+1}\setminus\{0\}$.
The second ingredient is given by Corollary \ref{Furino}.
The third ingredient is well-known. For instance, it is easily obtainable by combining Theorems 3 and  4 in \cite{D2}.

We note that the preceding proof immediately extends to the much stronger version of Result 1 mentioned above.
Also, all necessary ingredients can easily be written down explicitly, so the construction could be made as direct as desired.

\subsection{Comments on Result 2}

In their Remark 3, the authors of \cite{AMC} concede that the parameters of the DDF mentioned in Result 2 are not new, since they were already obtained in \cite{B} and \cite{Ge}. 
But, according to them, their construction has the advantage of being direct.
While it is true that the construction in \cite{Ge} is recursive, the one derivable from the proof of Theorem \ref{Buratti} given in \cite{B} is actually as direct as possible. 
We now show that the {\it explicit} construction can be presented in a dozen lines. Thus there is no need for several pages of generalized cyclotomy calculations.
As it is standard, $\F_q$ and $\F_q^*$ will denote the field of order $q$ and its multiplicative group. Also, $U(R_v)$ will denote the group of units of $R_v$.

\vspace{2mm}

\eject
\noindent
\emph{Explicit construction of the $(v,k,{k-1\over2})$-DDFs coming from Theorem \ref{Buratti}}.
 
Let $R_v=\F_{q_1}\times \dots \times \F_{q_t}$ with $q_i=2kn_i+1$. For $1\leq i\leq t$,
let $\omega_i$ be a primitive element of $\F_{q_i}$, and let $A$ be the subgroup of order $k$ of $U(R_v)$ 
generated by  $(\omega_1^{2n_1},\dots,\omega_t^{2n_t})$. For $1\leq i\leq t$, set $S_i=\{\omega_i^j \ | \ 1\leq j\leq n_i\}$.
The set ${\cal C}$ of the associate classes\footnote{Two elements $a$, $b$ of 
a ring with identity $R$ are {\it associates} if $b=au$ for a suitable unit $u$. Being associates is an equivalence relation whose equivalence classes
are called the {\it associate classes} of $R$.} of $R_v$ is  in one-to-one correspondence with the power-set of $\{1,\dots,t\}$: the members of ${\cal C}$
are precisely all sets of the form  $C=C_1 \times \dots \times C_t$ where each factor $C_i$ is either $\{0\}$ or $\F_{q_i}^*$. For every non-zero $C\in{\cal C}$, choose a non-null factor $C_i$ of it,
and let $\sigma(C)$ be the subset of $C$ of size ${|C|\over 2k}$ obtained from $C$ itself by replacing the chosen $C_i$ with the set $S_i$. 
Set $X=\displaystyle\bigcup_{C\in{\cal C}^*}\sigma(C)$ where ${\cal C}^*$ is the set of non-zero classes of ${\cal C}$. 
Then ${\cal F}=\{xA \ | \ x\in X\}$ is the desired $(v,k,{k-1\over2})$-DDF.

\bigskip\noindent
Example. Let $v=1729$ and $k=3$. We have $R_v=\F_{q_1}\times\F_{q_2}\times\F_{q_3}$ with $(q_1,q_2,q_3)=(7,13,19)$.  Thus
$q_i=2kn_i+1$ with $(n_1,n_2,n_3)=(1,2,3)$. Take the primitive elements $\omega_i$ as follows: $(\omega_1,\omega_2,\omega_3)=(3,2,2)$.
Then $A$ is the group of units of order 3 generated by the triple $(3^2,2^4,2^6)=(2,3,7)$, i.e., $A=\{(1,1,1),(2,3,7),(4,9,11)\}$.
The non-zero associate classes of $R_v$ are:
$$\F_7^*\times\{0\}\times\{0\};\quad  \{0\}\times\F_{13}^*\times\{0\};\quad \{0\}\times\{0\}\times\F_{19}^*;$$
$$\F_7^*\times\F_{13}^*\times\{0\};\quad  \F_7^*\times\{0\}\times\F_{19}^*;\quad \{0\}\times\F_{13}^*\times\F_{19}^*;$$ 
$$U(R_v)=\F_7^*\times\F_{13}^*\times\F_{19}^*.$$ 
Following the instructions of the proof given above, the related $\sigma(C)$ can be taken as follows:

$\{3\}\times\{0\}\times\{0\}=\{(3,0,0)\};$

$\{0\}\times\{(2,4)\}\times\{0\}=\{(0,2,0),(0,4,0)\};$

$\{0\}\times\{0\}\times\{2,4,8\}=\{(0,0,2),(0,0,4),(0,0,8)\};$

$\{3\}\times\F_{13}^*\times\{0\}=\{(3,i,0) \ | \ 1\leq i\leq 12\};$

$\{3\}\times\{0\}\times\F_{19}^*=\{(3,0,i) \ | \ 1\leq i\leq 18\}$;

$\{0\}\times\{2,4\}\times\F_{19}^*=\{(0,i,j) \ | \ i=2,4; 1\leq j\leq 18\};$

$\{3\}\times\F_{13}^*\times\F_{19}^*=\{(3,i,j) \ | \ 1\leq i\leq 12; 1\leq j\leq 18\}.$

If $X$ is the union of the above sets, then $\{x A \ | \ x\in X\}$ is a $(1729,3,1)$-DDF in $R_{1729}$.

It is worth noticing that there is a huge number of other possible choices for $X$. Only the method described above gives 24 possibilities
since we may change the choices of some related $\sigma(C)$. For instance, if $C$ is the fourth associate class 
$\F_7^*\times\F_{13}^*\times\{0\}$, as related $\sigma(C)$ we could also take $\{(i,j,0) \ | \ 1\leq i\leq 6; j=2,4\}$.

\subsection{Comments on Result 3}

In \cite{AMC}, Result 3 is stated in terms of a PDF with just one non-singleton block. 
This realizes what we have privately been fearing for some time: difference sets are translated to PDFs or ZDBFs. 
We strongly discourage to use this approach, as it is not helpful for the study of difference sets.
Research on these objects has a long tradition, and the extensive literature on difference sets contains a substantial number of deep results.
To get an idea of the complexity and wealth of the subject which had already been reached by 1999, 
the reader might have a look at the 170 page chapter on difference sets in \cite{BJL}. Needless to say, there has been considerable further progress since then.

Apart from these general reservations, Result 3 is unfortunately wrong. It actually holds only for the special case $(e,h)=(q-1,1)$,
which corresponds to the class of {\it Singer} difference sets -- probably the best known and truly classical examples of difference sets.
On the other hand, it is easy to see that it fails for all other pairs $(e,h)$. Assume otherwise.
Then we would have two difference sets whose parameter triples $(v,k,\lambda)$ and $(v\mu,k\mu,\lambda\mu)$ 
are distinct and proportional\footnote{One triple is $(v,k,\lambda)$ with $v={q^m-1\over q-1}$,  $k={q^{m-1}-1\over q-1}$, $\lambda={q^{m-2}-1\over q-1}$,
and the other triple is $(v,k,\lambda)$ multiplied by $\mu={h(q-1)\over e}$.}.
Let us show that this is never possible unless $(v,k,\lambda)$ is the trivial triple $(k,k,k)$. 

If $D$ is a $(v,k,\lambda)$-DS and $D'$ is a $(v\mu,k\mu,\lambda\mu)$-DS, then the trivial identities (obtained by counting differences in two ways)
$$\lambda(v-1)=k(k-1),\quad \lambda\mu(v\mu-1)=k\mu(k\mu-1)$$
hold. Dividing the second identity by $\mu$ and then dividing the first identity by the second, 
we obtain ${v-1\over v\mu-1}={k-1\over k\mu-1}$ which gives $(v-1)(k\mu-1)-(v\mu-1)(k-1)=0$.
Expanding the left hand side, we finally have $(v-k)(\mu-1)=0$ which means that either $v=k$, that is $(v,k,\lambda)=(k,k,k)$, or $\mu=1$.
(Of course, the same argument applies to symmetric designs in general.)

\vspace{1pt}
As an illustration of their Result 3, the authors of \cite{AMC} consider the case $q=m=4$, $e=3$ and $h=2$, see their Example 3.10.
They claim that this gives a $(170,42,10)$-DS, and they explicitly display a 42-subset $D$ of $\Z_{170} \cong  \Z_{85}\times\Z_2$,
inviting the reader to check its correctness by using a computer. 
However, the list $\Delta D$ of differences from $D$ actually covers every element of $\Z_{170} \setminus\{0,85\}$ exactly 10 times,
whereas the involution $85$ appears precisely 42 times as a difference.

It is really a pity that they did not realize that the triple $(170,42,10)$ is not admissible; in that case, they would have discovered this awful mistake. 
Actually, one does not even need a computer check to see that the example presented is not correct: 
just looking at the set $D$,  one quickly notices that it has the form $A \cup (A+85)$, where $A$ is a $21$-subset of $\Z_{85}$,
which immediately explains why $85$ appears precisely 42 times as a difference.

\section{Some remarks on divisible difference sets}

In this  final section, we wish to point out that it is easy enough, even if only marginally interesting, to give a corrected version of Result 3.
The incorrect specific Example 3.10 in \cite{AMC} in reality gives a divisible difference set (DDS) with parameters
$$m=85,\; n=2,\; k = \lambda_1 = 42 \; \mbox{and}\; \lambda_2 = 10$$
in $\Z_{170}$. Let us recall the definition of these objects.

A \emph{divisible difference set} with parameters $m$, $n$, $k$, $\lambda_1$ and $\lambda_2$ (for short, an $(m,n,k,\lambda_1,\lambda_2)$-DDS) 
in an (additively written) group $G$ of order $mn$ relative to a normal subgroup $N$ of order $n$ is a $k$-subset $D$ of $G$ 
such that every element $g \in G \setminus N$ has exactly $\lambda_2$ representations as a difference $g = d-d'$ with $d,d' \in D$,
whereas  every element $g\neq 0$ in $N$ has exactly $\lambda_1$ such representations. 
Again, such objects have been studied quite extensively, with particular emphasis on \emph{relative difference sets}, that is, the special case $\lambda_1=0$.

We suggest the old paper \cite{D3} by the second author as an extended introduction to various aspects of this area.
In particular, it should be noted that divisible difference sets correspond to square divisible designs with a Singer group, 
just as ordinary difference sets correspond to symmetric designs with a Singer group.
Using the DDS terminology, the corrected version of Result 3  reads as follows:
\begin{quote}
{Result 3*.} For any prime power $q\equiv1$ $($mod $e)$, any $d \geq3$ coprime with $e$, and any $h$ in the
closed interval $[1,e]$, there exists an $(m,n,k,\lambda_1,\lambda_2)$-DDS in $\Z_{{q^d-1\over e}}\times \Z_h$ with 
\noindent\begin{center}
$m={q^d-1\over q-1};\;\; n={h(q-1)\over e}; \;\; k= \lambda_1 = {h(q^{d-1}-1)\over e}, \;\; \lambda_2={h(q^{d-2}-1)\over e}.$
\end{center}
\end{quote}
Again, the use of generalized cyclotomy just leads to a lot of entirely superfluous restrictions, while simultaneously obscuring the true reason behind the result. 
In fact, Result 3* is an extremely special case of the following simple observation (applied to the classical Singer difference sets).

\begin{thm}\label{DDS}
If there exists an $(m,k,\lambda)$-DS in $G$, then there also exists an $(m,h,kh,kh,\lambda h)$-DDS in $G\times H$,
where $H$ may be any (even non-abelian) group of order $h$.
\end{thm}

\begin{proof}
Let S be any $(m,k,\lambda)$-difference set in $G$ and put $D := S \times H$.  It is straightforward to check 
(either directly, or via a short group ring computation) that $D$ is the desired DDS  in $G\times H$. 
\end{proof}

Actually, Theorem \ref{DDS} is not really interesting, even if it should be new, as it is just the difference set version of an old result on divisible designs.
In general,  divisible designs with $r = \lambda_1$ (that is, $k = \lambda_1$ for the square case in which we are) are said to be \emph{singular}.
A 1952 result by Bose and Connor \cite{BC} states that all singular divisible designs are equivalent to 2-designs, with each point taken $n$ times;
in the difference set setting, this  translates into the construction just presented. 
Perhaps this has never been stated explicitly, because singular divisible designs are usually considered to be trivial, but it is at least folklore.
We note that virtually every paper on divisible designs immediately excludes the singular case, for obvious reasons.

\vspace{1mm}
Of course,  divisible difference sets are not  pertinent to the topic of PDFs and ZDBFs anyway, as they have two distinct $\lambda$-values.
Let us express our fervent hope that no new research trend will be started by using generalized cyclotomy to construct ``PDFs (or ZDBFs) with more than one $\lambda$-value''.
Unfortunately, in view of our experiences over the past years, we cannot  be entirely confident about this.

\end{document}